\documentclass[a4paper]{article} 

\usepackage{hhline}
\usepackage{caption}
\usepackage{booktabs}

\usepackage{amsfonts}
\usepackage{amsmath} 
\usepackage{amssymb}
\usepackage{graphicx,wrapfig}
\usepackage{array}
\newcolumntype{C}[1]{>{\centering\let\newline\\\arraybackslash\hspace{0pt}}m{#1}}
\newcolumntype{L}[1]{>{\raggedright\let\newline\\\arraybackslash\hspace{0pt}}m{#1}}

\usepackage{amsthm}
\usepackage{dsfont}    
\usepackage{microtype} 

\usepackage{caption}
\usepackage{subcaption}
\usepackage{color}
\usepackage{url}

\usepackage[utf8]{inputenc}
\usepackage[english]{babel}

\newtheorem{theorem}{Theorem}
\theoremstyle{definition}
\newtheorem{definition}{Definition}
\newtheorem{lemma}{Lemma}
\newtheorem{corollary}[lemma]{Corollary}

\theoremstyle{remark}
\newtheorem*{remark}{Remark}
\theoremstyle{conjecture}

\newtheorem*{acknowledgement}{Acknowledgement}

\DeclareMathOperator{\sinc}{sinc}

\newcommand{\onehalf}{\frac{1}{2}}
\newcommand{\N}{\mathbb{N}}
\newcommand{\R}{\mathbb{R}}

\newcommand{\Z}{\mathbb{Z}}

\newcommand{\Cone}{\mathcal{C}^1}

\newcommand{\ra}{\rightarrow}
\newcommand{\mt}{\mapsto}

\newcommand{\sd}{\mathcal{D}}

\newcommand{\fatone}{\mathds{1}}
\newcommand{\dx}{\ \mathrm{d}x}

\newcommand{\dt}{\ \mathrm{d}\theta}

\newcommand{\gl}{\mathrm{GL(\R^2)}}
\newcommand{\s}{\mathbb{S}^2}

  \date{}
\begin{document}
	\title{Spherical cap discrepancy of perturbed lattices under the Lambert projection}
	\author{Damir Ferizovi\'{c}\thanks{ KU Leuven, Department of Mathematics, Celestijnenlaan 200b, Box 2400, 3001 Leuven, Belgium. Tel:+32 16 37 62 87, damir.ferizovic@kuleuven.be. Keywords: spherical cap discrepancy, two sphere, Lambert map, lattices.  MSC2020: 11K38 (primary), 52C99, 52A10.
			The author thankfully acknowledges support by the  Methusalem grant of the Flemish Government. }}
\maketitle 
\begin{abstract}
Given any full rank lattice $\Lambda$ and a natural number $N$, we regard the point set $\Lambda/N\cap(0,1)^2$  under the Lambert map to the unit sphere, and show that its  spherical cap discrepancy is at most of order $N$, with leading coefficient given explicitly  and depending on $\Lambda$ only. The proof is established using a lemma that bounds the amount of intersections of certain curves with fundamental domains that tile $\R^2$, and even allows for local perturbations of $\Lambda$ without affecting the bound, proving to be stable for numerical applications.  A special case yields the smallest constant for the leading term of the cap discrepancy for deterministic algorithms up to date.

\end{abstract}

\section{Introduction and main result}

The problem of distributing $N$-many points $P_N=\{p_1,\ldots,p_N\}$ uniformly on a sphere is well known and has applications in numerical integration, approximation, cartography and the applied sciences, see \cite{Whyte}, \cite{KuijlaarsSaff} and \cite{BrauchartGrabner} for some reviews which span over a century of  development.

One distinguishes between point sequences and point sets: in the first case $P_N$ consists of the initial $N$ elements of a sequence, and 
in the latter case  $P_N$ usually can  only be constructed for infinitely many $N$, with no relation from one admissible $N$ to the next.


There are a few notions to quantify how well distributed $P_N$ is: one example is the notion of energy, which is the sum of the $N(N-1)$-many evaluations of  pair-interactions under a given lower semi-continuous potential $F$, where for $F(x,y)=\|x-y\|^{-s}$ one speaks of the Riesz $s$-energy and minimizers of this potential have been intensively investigated, see for instance the comprehensive monograph by Borodachov, Hardin and Saff \cite{BorodachovSaffHardin}.
Another notion is that of seperation distance, i.e. the minimum of the pairwise distances of distinct elements among $P_N$. 
Our main result uses the notion of spherical cap discrepancy. 

\begin{definition}
	Let $P_N=\{p_1,\ldots,p_N\}\subset \s$. 
	The spherical cap discrepancy is
	$$\sd(P_N)=\sup_{w\in\s}\sup_{-1\leq t\leq1}\Bigg| \frac{1}{N}\sum_{j=1}^{N}
	\chi_{C(w,t)}(p_j)-\sigma(C(w,t))\Bigg|. $$
\end{definition}

Here $\s\subset \R^3$ is the unit two-sphere, i.e. the set of all unit vectors with norm induced by the standard dot product $\langle \cdot,\cdot \rangle$. A spherical cap is the set
$$C(w,t)=\big\{x\in\s\ |\ \langle w,x\rangle\geq t  \big\},$$
with center $w\in\s$ and height $t\in(-1,1)$. The boundary of a set $S$ is denoted by $\partial S$, thus 
$$\partial C(w,t)=\big\{x\in\s\ |\ \langle w,x\rangle=t  \big\},$$
and $\sigma$ is the normalized surface measure of $\s$, such that  $\sigma\big(C(w,t)\big)=\onehalf(1-t)$. The characteristic function of a set $A$ is $\chi_A$, i.e. $\chi_A(x)=1$ iff $x\in A$, and $0$ else. 

\vspace{0.3cm}

A reason why discrepancy is useful comes from the following relation
$$\Bigg|\frac{1}{N}\sum_{j=1}^Nf(p_j)-\int_{\s} f\ \mathrm{d}\sigma\Bigg|\leq \hat\sd(P_N)\mathcal{V}(f), $$
where $\hat\sd$ is a notion of discrepancy (not restricted to spherical cap) and $\mathcal{V}(f)$ is a constant depending on the function $f$, see \cite{GrabnerKlingerTichy}, \cite{ChoiratSeri}. Such inequalities often bear the name Koksma--Hlawka in honour of the discoverers.

This paper deals  exclusively with point sets on the two-sphere, which are derived as perturbations of any full rank lattice under a specific equal area map and thus adding to the list of known constructions with good distribution properties -- an excerpt of which is found below for the readers convenience, where we state the origin and place of investigation to the best of our knowledge.

\begin{table}[htb]
	\centering
	\begin{tabular}{l*{2}{c}}
		\toprule
		\textit{Probabilistic algorithms} & \textit{Origin} & \textit{Invstigtd.} \\
		\midrule
		Uniformly at random & \cite{Cook} & \cite{Aistleitner},  \cite{BrauchartReznikovSaffSloanWang} \\
		
		Ensembles based on determinantal point processes  & \cite{Krishnapur}&
		\cite{AlishahiZamani},  \cite{BeltranMarzoOrtega}       \\
		
		Roots of random polynomials in $\R^2\ra\s$  & \cite{BogomolnyBohigasLeboeuf}              & \cite{ArmentoBeltranShub}     \\
		
		The diamond ensemble ({\small can also be made deterministic})   &\cite{BeltranEtayo} & \cite{BeltranEtayo}, \cite{Etayo}   \\ 
		
		Point sets based on	jittered sampling   &\cite{Bellhouse}, \cite{Beck2} & \cite{Beck}  \\ 
		\bottomrule

		\toprule
		\textit{Deterministic algorithms} & \textit{Origin} & \textit{Invstigtd.} \\
		\midrule
		Spiral points & \cite{Rakhmanov}, \cite{Bauer} & \cite{Hardin}  \\
		
		Hierarchical, Equal Area and iso-Latitude Pixelation  & \cite{Gorski}&
	\cite{Hardin}, \cite{FerHofMas}     \\
		
		Point sets using group action  &  \cite{LubotzykPhillipsSarnak}              &  \cite{LubotzykPhillipsSarnak}    \\
		
		Spherical Fibonacci lattice   & \cite{Aistleitner} & \cite{Aistleitner} \\ 
		
		Spherical Fibonacci grid   &  \cite{SwinbankPurser} & \cite{Hardin}  \\ 
		\bottomrule
	\end{tabular}
\end{table}

 We will prove a bound on the spherical cap discrepancy of our point constructions which is of the same order as the best results obtained so far for deterministic point sets found in the literature. An open and hard problem is to improve these bounds to the order obtained in a fundamental result due to Beck  in \cite{Beck, Beck2}, where we find the existence of points $P_N^{\star}$ and  constants  independent of $N$, such that 
\[cN^{-3/4} \le \sd(P_N^{\star}) \le CN^{-3/4}\sqrt{\log N}.\]
The construction of Beck is probabilistic and an algorithm to generate random point sets with discrepancy matching this upper bound with high probability was  obtained for instance in \cite{AlishahiZamani}. 

\vspace{0.2cm}

Bounds for deterministic  point sets are usually hard to derive and were first achieved by Lubotzky, Phillips and Sarnak in \cite{LubotzykPhillipsSarnak} with an upper bound of the order $\log(N)^{2/3}N^{-1/3}$.
Aistleitner, Brauchart, and Dick showed in \cite{Aistleitner}  that for spherical digital nets and spherical Fibonacci lattices the spherical cap discrepancy is upper bounded by an order of $N^{-1/2}$. Further bounds were given for the diamond ensemble  by Etayo in \cite{Etayo} and for HEALPix generated points by Hofstadler, Mastrianni and the author in  \cite{FerHofMas},  where both point sets are shown to have a spherical cap discrepancy of order $N^{-1/2}$. 

Etayo opened the contest to find  deterministic point sets $P_N$ which allow for the smallest constant bounding $\sqrt{N}\sd(P_N)$, where she took the lead in \cite{Etayo}  with the Diamond ensemble and a bound of  $4+2\sqrt{2}$. We will identify a family of point sets  where $\sqrt{18}$ will do, see Section \ref{subsec_standardlattice}.

We construct point sets  based on the method deployed in \cite{Aistleitner}, which originated in \cite{CuiFreeden}. Let $\gl$ be the group of invertible $2\times 2$ matrices acting on $\R^2$ with identity matrix $\fatone$. The Frobenius norm of the matrix $Q$ is  $\|Q\|_F=\sqrt{\mathrm{trace}\ Q^tQ}$, the Lebesgue measure on $\R^2$ is $\lambda$, and $e_1,e_2$ denote the standard basis of $\R^2$.
\begin{definition}\label{def_pointsPNtilingTN}
	A lattice $\Lambda^Q$ is a set $Q\Z^2$ for some $Q\in\gl$. The associated scaled tiling via a fundamental domain  $\Omega^Q=Q[0,1)^2\subset\R^2$ and $K\in\N$ is 
	\begin{equation*}
	\R^2=\bigcup_{p\in \Lambda^Q}T_K(p),\hspace{0.3cm}\mbox{where}\hspace{0.3cm}T_K(p)=\frac{1}{K}\Big(p+ \Omega^Q\Big).
	\end{equation*}
	 Let $I^2=[0,1)\times(0,1)$ denote the unit square with three boundary sides removed, then for fixed $K\in\N$ we choose arbitrary $z_K^p\in T_K(p)$ for $p\in\Lambda^Q$ and define a point set  
	$$ P^Q(K)=\big\{z_K^p\ |\ p\in \Lambda^Q\big\}\cap I^2.$$
\end{definition}
We note that the Lambert map $L:I^2\ra\s$, which we introduce later,  identifies two sides of $\partial I^2$ and maps two sides to the poles. This identification has an effect on the discrepancy, and to describe   it we	define a parametrization $\rho:[0,4]\ra\partial I^2$ via $f(x):=\max\big(\min(x,1)+\min(2-x,0),0\big):$
\begin{equation}\label{eq_parametrizationI2}
 x\mapsto \rho(x)\ =\ f(x)e_1+f(4-x)e_2.
\end{equation} 
Before we state our main result, we need to introduce three terms, which are bounded above by constants depending on $Q$ only, as we will show:
\begin{description}
	\item[1)] Let $N^Q(K)=\#P^Q(K)$ be the amount of points in $P^Q(K)$, and set
	$$d^Q(K)=\frac{1}{K}\Big|N^Q(K)-\frac{K^2}{|\det(Q)|}\Big|.$$
	
	\vspace{-0.3cm}
	
	\item[2)]  Set 
	\vspace{-0.2cm}
	$$C_L^Q=\sup_{w\in \s}\sup_{-1\leq t\leq1} \mathrm{length}\Big(Q^{-1}\big(L^{-1}\big(\partial C(w,t)\big)\big)\Big).$$
	
	\item[3)] 
	For $S\subset [0,1]$ and $f$ as in \eqref{eq_parametrizationI2}, we define $R(S)=\{x\in[0,4]\ |\  f(4-x)\in S\}$ and 
	\vspace{-0.1cm}
	$$Ad(S,K)=\big\{p\in \Lambda^Q\ |\ T_K(p)^o\cap\rho(R(S))\neq \emptyset \big\},$$
	the collection of points $p$ such that the interior of $T_K(p)$ intersects $\partial I^2$ at $\rho(R(S))$. The supremum of the marginal discrepancy around $\partial I^2$ then is
	$$M^Q(K)=\sup_{0\leq a<b\leq 1}\frac{K}{|\det(Q)|}\ \bigg|\sum_{p\in Ad([a,b],K)}\Big(\frac{\chi_{P^Q(K)}(z^p_K)}{N^Q(K)}	-\lambda\big(T_K(p)\cap I^2\big)\Big)\bigg|.$$
\end{description}
The term $M^Q(K)$ bounds the discrepancy that arises around $L(\partial I^2)$, while taking the geometry of images of spherical caps under $L^{-1}$ into account.

\begin{theorem}[Main Result]\label{thm_main}
	Let $Q\in\gl$, $K\in\N$ and $P^Q(K)$ be chosen as in Definition \ref{def_pointsPNtilingTN} with $N=N^Q(K)$, then 
	$$\sd(L(P^Q(K)))\ \leq\ \Big(d^Q(K)+\sqrt{2}\cdot C_L^Q+M^Q(K) \Big)\frac{\sqrt{|\det(Q)|}}{\sqrt{N}}+O(N^{-1}).$$
\end{theorem}
\begin{remark}
	
	A possible choice for $P^Q(K)$ is $(K^{-1}\Lambda^Q)\cap I^2$, and Theorem \ref{thm_main} shows that small numerical inaccuracies in representation of these elements do not affect the bound on discrepancy.
\end{remark}
Sometimes a less precise but more succinct expression is more desirable:

\begin{corollary}\label{cor_main}
	Let $Q\in\gl$, $K\in\N$ and $P^Q(K)$ be chosen as in Definition \ref{def_pointsPNtilingTN} with $N=N^Q(K)$, then 
		$$\sd(L(P^Q(K)))\ \leq\ \frac{\|Q\|_F}{\sqrt{|\det(Q)|}}\frac{8+3\sqrt{2}}{\sqrt{N}}+O(N^{-1}).$$
\end{corollary}
\begin{proof}
	This follows from Theorem \ref{thm_main} with  Lemma \ref{lem_boundonCL}, Lemma \ref{lem_lengthboundcurvematrix}, Lemma \ref{lem_differenceKandN2} and the remark after it. 
\end{proof}
It is easy to see from the proof of Lemma \ref{lem_lengthboundcurvematrix}  that $\sqrt{2}\sqrt{|\det(Q)|}\leq\|Q\|_F$.
If one is willing to adjust $P^Q(K)$ to $P_{mod}^Q(K)$ with the algorithm described in Section \ref{subsec_pointmod}, then an improved upper bound holds for this modified point set.
\vspace{0.3cm}

%

\textbf{Outline of the paper.} In Section \ref{sec_proofs} we first state the definition of the Lambert equal area map and prove a bound on $C_L^\fatone$. Next we  introduce the notion of $n$-convex curve, which allows us to prove in  Lemma \ref{lem_sharplengthbound} a  bound of how many fundamental domains of $\frac{1}{K}\Lambda^Q$ are intersected by it -- this result constitutes the backbone of this work. The section then ends with a bound on $d^Q(K)$ and a proof of Theorem \ref{thm_main}.
In Section \ref{sec_applications} we apply Theorem \ref{thm_main} to specific cases, this way we obtain many deterministic point sets with the least constant for the leading term  up to date.

\section{Intermediate results and proof of Theorem \ref{thm_main}}\label{sec_proofs}

\subsection{The Lambert map}
By $\s_p$ we denote the punctured unit sphere, i.e. $\s$ with the poles removed. 
\begin{definition}  
	The Lambert map $L$ is a well known area preserving map, i.e. for open sets $U\subset I^2$ and the Lebesgue measure $\lambda$, we have $\lambda(U)=\sigma(L(U))$, where
		$$
	\begin{array}{rcl}
	L:I^2\hspace{-.2cm} &\ra& \s_p\\
	\ &\ & \vspace{-.3cm} \\
	(x,y)\hspace{-.2cm}&\mt&\hspace{-.2cm}\left(\hspace{-0.2cm}\begin{array}{l}
	2\sqrt{y-y^2}\cos(2\pi x)\\
	2\sqrt{y-y^2}\sin(2\pi x)\\
	1-2y
	\end{array}\hspace{-0.2cm}\right),
	\end{array}
	$$
	and the inverse  is given in terms of the standard parametrization\footnote{Here $0\leq\phi<2\pi$ and $0< \theta< \pi$.} of $\s$	
	\begin{equation}\label{eq_inverseLambert}
	\begin{array}{rcl}
	L^ {-1}:\s_p&\ra& \hspace{-.2cm}I^2 \\
	\ &\ &\vspace{-.3cm} \\
	\left(\hspace{-0.2cm}\begin{array}{l}
	\cos(\phi)\sin(\theta)\\
	\sin(\phi)\sin(\theta)\\
	\cos(\theta)
	\end{array}\hspace{-0.2cm}\right)\hspace{-.2cm}&\mt& \hspace{-.2cm}
	\frac{1}{2\pi}\left(\hspace{-0.2cm}\begin{array}{c}
	\phi\\
	\pi(1-\cos(\theta))
	\end{array}\hspace{-0.2cm}\right).
	\end{array}
	\end{equation}
The map $L$ clearly extends to all of $\partial I^2$, with the obvious caveats. 
\end{definition}
 
\begin{lemma}\label{lem_boundonCL}
	  $C_L^\fatone\ \leq\ 3$.
\end{lemma}
\begin{proof}
Note that the  map $L^{-1}$  projects points from the sphere, parallel to the plane $\R^2\times\{0\}$, to the enveloping cylinder (or rather cylindrical surface). The cylinder is then "cut", "rolled" out flat and scaled to a square of area 1. In order to bound the length of the curve $L^{-1}(\partial C\cap \s_p)$, it will be enough to work with the projection on the enveloping cylinder, where we make following  observation that we consider obviously true: 

 Let $z_1(w,t)<1,\ z_2(w,t)>-1$ denote the maximal and minimal $z$-value of a spherical cap $C(w,t)$ in cylindrical coordinates. If $1>z_1(w',t')\geq z_1(w,t)$ and $-1<z_2(w',t')\leq z_2(w,t)$, then the length of $\partial C(w',t')$ is bigger or equal the one of $\partial C(w,t)$, and  the length of $\partial C(w,t)$ is bounded by the length of any cap $C$ with the same maximal/minimal $z$-value that contains exactly one of the poles of $\s$. This  stays true under the aforementioned scaling to the unit square.
 
Thus in order to maximize the length we let $z_1,z_2$ approach $\pm 1$ and choose $w$ so that exactly one of the poles is contained in the cap, i.e. regard $C(w,0)$ with
$$w=\big(\sin\big(\tfrac{\pi}{2}-\epsilon\big),0,\cos\big(\tfrac{\pi}{2}-\epsilon\big)\big)^t$$ and $\epsilon>0$ small.
A parametrization for $\partial C(w,0)$ is then given by 
	$$
\begin{array}{rcl}
\varphi_{\pm}:[\epsilon, \pi-\epsilon] &\ra& [0,2\pi]\\
\ &\ & \vspace{-.3cm} \\
\theta&\mt&\pm\arccos\big(-\cot(\theta)\cot\big(\frac{\pi}{2}-\epsilon\big)\big);
\end{array}
$$
such that
$$
\begin{array}{rcl}
	\gamma_{+}:[\epsilon, \pi-\epsilon] &\ra& \partial C(w,0)\\
	\ &\ & \vspace{-.3cm} \\
	\theta&\mt&\hspace{-0.3cm}\left(\begin{array}{l}
		\cos\big(\varphi_{+}(\theta)\big)\sin(\theta)\\
		\sin\big(\varphi_{+}(\theta)\big)\sin(\theta)\\
		\cos(\theta)
	\end{array}\hspace{-0.2cm}\right)
\end{array}
$$ 
defines a curve that runs from $z_1$ to $z_2$ along $\partial C(w,0)$, and $\gamma_{-}$ defines the other half.
Since $L^{-1}\circ\gamma_{+}$ and $L^{-1}\circ\gamma_{-}$ have the same length, it is enough to consider just one of these curves, bound its length and multiply by 2: thus we regard the curve
$$\frac{1}{\pi}\left(\hspace{-0.2cm}\begin{array}{c}
\arccos\big(-\cot(\theta)\cot\big(\frac{\pi}{2}-\epsilon\big)\big)\\
\pi(1-\cos(\theta))
\end{array}\hspace{-0.2cm}\right),
$$
where we already multiplied by a factor of $2$ in equation \eqref{eq_inverseLambert}. Thus we have to bound
$$
 \mathrm{length}=\int_{\epsilon}^{\pi-\epsilon}\sqrt{\frac{1}{1-\cot(\theta)^2\cot\big(\frac{\pi}{2}-\epsilon\big)^2}\frac{\cot\big(\frac{\pi}{2}-\epsilon\big)^2}{\pi^2\sin(\theta)^4}+\sin(\theta)^2}\dt.
$$
We can simplify the expression in the square root, first with the angle sum formula for the cosine (with $\epsilon <\theta<\pi-\epsilon$) to obtain
$$ \frac{-1}{\cos\big(\frac{\pi}{2}-\epsilon-\theta\big)\cos\big(\frac{\pi}{2}-\epsilon+\theta\big)}\frac{\cos\big(\frac{\pi}{2}-\epsilon\big)^2}{\pi^2\sin(\theta)^2}+\sin(\theta)^2,$$
and then with $\cos\big(\epsilon-\frac{\pi}{2}\big)=\sin(\epsilon)$ and a symmetry argument to get
$$
\mathrm{length}=2\int_{\epsilon}^{\frac{\pi}{2}}\sqrt{\frac{1}{\sin(\epsilon+\theta)\sin(\theta-\epsilon)}\frac{\sin(\epsilon)^2}{\pi^2\sin(\theta)^2}+\sin(\theta)^2}\dt.
$$
Let  $\sinc(x)=\frac{\sin(x)}{x}$, then for any choice of $a\in(0,\frac{\pi}{2}]$, we have
$$\sinc(a)\ x\leq \sin(x)\leq x, \mbox{ for }x\in[0,a];$$
which we use for small $\epsilon$ and $0<2\epsilon<a<\frac{\pi}{2}$ to obtain
\begin{align*}
	\mathrm{length}&\leq\frac{2}{\sinc(a)^2\pi}\int_{\epsilon}^{a-\epsilon}\sqrt{\frac{1}{\theta^2-\epsilon^2}\frac{\epsilon^2}{\theta^2}}\dt+
	\frac{\epsilon\cdot\pi^2}{4\sqrt{a(a-2\epsilon)^3}}
	+
	2\int_{\epsilon}^{\frac{\pi}{2}}
	\sin(\theta)\dt\\
	&=\frac{2}{\sinc(a)^2\pi}\int_{1}^{a/\epsilon-1}\sqrt{\frac{1}{x^2-1}}\frac{1}{x}\dx+
	O(\epsilon)+	2\overset{\lim \epsilon\ra 0}{=}\frac{2}{\sinc(a)^2\pi}\frac{\pi}{2}+	2,
\end{align*}
where we used that the length increases with decreasing $\epsilon$, and that
$$\frac{\mathrm{d}}{\mathrm{d}x}\arctan\big(\sqrt{x^2-1}\big)= \sqrt{\frac{1}{x^2-1}}\frac{1}{x}.
$$
The inequality above is thus true for any $0<a$, and hence also for  $\sinc(0)=1$, which proves the claim.
%
%
\end{proof}

The next lemma will help us relate $C^Q_L$ to $C^\fatone_L$ and should be known, but a reference could not be found.

\begin{lemma}\label{lem_lengthboundcurvematrix}
	Given a matrix $A\in\R^{2\times 2}$ and a finite length $\Cone$-curve $\beta$ in $\R^2$, then for $\ell_\beta=\mathrm{length}(\beta)$,  $\ell_{A\beta}=\mathrm{length}(A\beta)$, $m_A=\min\{\|Ae_1\|,\|Ae_2\|\}$ and $M_A=\max\{\|Ae_1\|,\|Ae_2\|\}$ we have
	$$\begin{array}{rlcl}
	&\ell_{A\beta}=\ell_p& \hspace{0.3cm}\mbox{if}& A \mbox{ has rank }1,\\
	m_A\cdot\ell_\beta\leq\hspace{-0.2cm}&\ell_{A\beta}\leq M_A\cdot\ell_\beta& \hspace{0.3cm}\mbox{if}& A \mbox{ has rank }2\mbox{ and $A$  orthogonal},\\
	&\ell_{A\beta}\leq\|A\|_F\cdot\ell_\beta& \hspace{0.3cm}\mbox{if}& A \mbox{ has rank }2,\\
	\end{array} $$
	where $\ell_{p}=\|Av\|\cdot \big(\max\{\langle \beta,v\rangle\}-\min\{\langle \beta,v\rangle\}\big)$ with $\|v\|=1$ and $v\perp \ker(A)$.
\end{lemma}
\begin{proof} Let $\beta(t)=\big(x(t),y(t)\big)$ be an injective $\Cone$-parametrization with $0<t<1$ (thus differing from the geometric curve at most by two endpoints)  and assume $Ae_1\neq 0$. Set
	$$A=\begin{pmatrix}
	a & b\\
	c& d
	\end{pmatrix},
	\hspace{0.3cm}\mbox{and}\hspace{0.3cm}
	\mathbf{R}=\frac{1}{\|(a,c)\|}
	\begin{pmatrix}
	a & -c\\
	c& a
	\end{pmatrix}$$ 
	with $\sqrt{a^2+c^2}=\|(a,c)\|$. We define $\alpha=\langle Ae_1,Ae_2\rangle\cdot\|(a,c)\|^{-2}$, thus 
	\begin{equation}\label{eq_decompositionofA}
	\frac{\mathbf{R}^{t}}{\|(a,c)\|}\binom{a}{c}=\binom{1}{0},  \hspace{0.3cm}\mbox{and}\hspace{0.3cm}
	\frac{\mathbf{R}^{t}}{\|(a,c)\|}\binom{b}{d}=\binom{\alpha}{\delta}\hspace{0.2cm}\mbox{with }\delta=\frac{\det(A)}{\|(a,c)\|^2}.
	\end{equation}
	Hence we have
	$$
	A=
	\mathbf{R}\|(a,c)\|\begin{pmatrix}
	1& \alpha\\
	0& \delta 
	\end{pmatrix}.$$

	\textbf{The case $\delta=0$}. By assumption there is $v,w\in\R^2$ with $\|v\|,\|w\|=1$ such that $\|Av\|>0$, $\|Aw\|=0$ and $v\perp w$. Choose a rotation matrix $T$ such that $Te_1=v$ (thus $Te_2=\pm w$) and proceed as above to obtain with $A'=AT$
	$$A'\beta(t)=\mathbf{R}'\|A'e_1\|x(t)e_1.$$
	Thus using $\hat\beta=T^{-1}\beta$ we obtain our claim from
	$$\mathrm{length}(A\beta)=\mathrm{length}(A'\hat \beta)=\|Av\|\big(\max\{\langle \hat\beta,e_1\rangle\}-\min\{\langle \hat\beta,e_1\rangle\}\big).$$
	
		\textbf{The case $\delta\neq0$}. Let $C$ be defined as $C=\delta^2+\alpha^2+1$ if  $\alpha\neq 0$ and  $C=\max\{\delta^2,1\}$ else, then we have  for $\dot{x}^2+\dot{y}^2>0$ (and with $2\alpha\dot{x}\dot{y}\leq \alpha^2\dot{x}^2+\dot{y}^2$)
	$$ \frac{[ \dot{x}+\alpha \dot{y}]^2+\delta^2\dot{y}^2}{\dot{x}^2+\dot{y}^2}=
	\frac{\dot{x}^2+2\alpha \dot{x} \dot{y}+ (\alpha^2+\delta^2 )\dot{y}^2   }{\dot{x}^2+\dot{y}^2}\leq C.$$
	Thus
	\begin{align*}
	\ell_{A\beta}&=\|(a,c)\|
	\int_{0}^{1}\sqrt{\Big\langle\mathbf{R} \binom{\dot{x}+\alpha \dot{y}}{\delta \dot{y}},\mathbf{R} \binom{\dot{x}+\alpha \dot{y}}{\delta \dot{y}}  \Big\rangle}\ \mathrm{d}t\\
	&=\|(a,c)\|
	\int_{0}^{1}\sqrt{\frac{[ \dot{x}+\alpha \dot{y}]^2+\delta^2\dot{y}^2}{\dot{x}^2+\dot{y}^2}}
	\sqrt{ \dot{x}^2+\dot{y}^2}\ \mathrm{d}t\ \leq\ \|(a,c)\|	\sqrt{C}\ \ell_\beta.
	\end{align*}
	Note that by \eqref{eq_decompositionofA} we have $\delta^2+\alpha^2=\|(\alpha,\delta)\|^2=\|(b,d)\|^2/\|(a,c)\|^2$, so that 
	$$ \delta^2+\alpha^2+1=\frac{\mbox{trace}\big(A^tA\big)}{\|(a,c)\|^2},$$
	and if $\alpha=0$, it follows that  $\|(a,c)\|	\sqrt{C}=\max\{\|(a,c)\|,\|(b,d)\|\}$. 
	Finally note that in the latter case we have $\min\{\delta^2,1\}\big(\dot{x}^2+\dot{y}^2\big)\leq \ \dot{x}^2+\delta^2\dot{y}^2$, and we proceed as above to finish the proof.
\end{proof}

\subsection{Intersection of convex curves with fundamental domains}
In this section we will introduce the notion of $n$-convex curve, for which we can prove sharp bounds on how many fundamental domains are intersected by it. This will then give the boundary contribution in the proof of Theorem \ref{thm_main}.
	\begin{definition}\label{def_nconvex}
	A continuous curve $\beta:[a,b]\ra \R^2$ ($a,b\in\R$, $a<b$) is $n$-convex for $n\in\N$, if there are points $a=t_0<t_1<\ldots<t_{n-1}<t_{n}=b$ with the property that  $\beta\big([t_j,t_{j+1}]\big)$ is a convex curve  for each $0\leq j< n$, i.e. there exist convex sets $A_1,\ldots,A_n$ with $\beta\big([t_j,t_{j+1}]\big)\subset \partial A_j$.
\end{definition}
A circle is $1$-convex, as is a straight line segment. Finite spirals (in length and winding number)  are $n$-convex. It seems that every (finite) connected part of the boundary of a pseudo-convex set, as introduced in \cite{Aistleitner}, is $n$-convex and vice versa.   This characterization is not necessary for our purposes, so we do not try to prove it here.

\begin{corollary}\label{cor_AistleitnerEtAl}
	The images $L^{-1}\big(\partial C(w,t)\big)$ for all $(w,t)\in\s\times(-1,1)$ are at most three distinct curves, each is at most $7$-convex, and they are not self-intersecting.
\end{corollary}
\begin{proof}
	This follows from  
	the analysis of Section 6 in \cite{Aistleitner}, where the authors show that there is an admissible covering by pseudo-convex sets of at most 7 parts (each part has some arc of the curve contained in its boundary).
\end{proof}

\begin{figure}[h]
	\centering
	\includegraphics[width=0.9\linewidth]{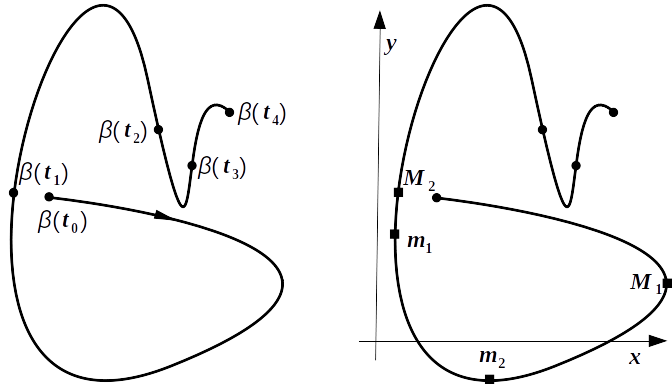}
	\caption*{\textbf{Fig. 1:} A 4-convex curve to illustrate  the proof of Lemma \ref{lem_sharplengthbound}.}
\end{figure}

\begin{definition}
	Given a  lattice $\Lambda^Q$, a continuous curve $\beta:[a,b]\ra\R^2$ and $K\in\N$. Let $\Omega^Q=Q[0,1)^2$, then the intersection number of $\beta$ with the tiling   $\frac{1}{K}\big(\Lambda^Q+ \Omega^Q\big)$ is 
	$$I_\beta^Q(K)=\#\big\{p\in\Lambda^Q\ |\ \big(\tfrac{1}{K}\Omega^Q+\tfrac{1}{K}p\big)\cap\beta\big([a,b]\big)\neq\emptyset \big\}.$$
\end{definition}

\begin{lemma}\label{lem_sharplengthbound}
	Let $\beta$ be a piece-wise $\Cone$-curve in $\R^2$,  $n$-convex with $m$-many self-intersections for $n,m\in\N_0$. Let $\Lambda^Q$ be a full rank  lattice, then 
	$$I_\beta^{Q}(K)\leq \sqrt{2}\cdot K\cdot \mathrm{length}\big(Q^{-1}\beta\big)  +19n-m+1. $$
\end{lemma}
\begin{proof} 
	Let $\beta:[0,1]\ra\R^2$ be a parametrization, and we can assume it to be $\Cone$ by cutting $\beta$ into sub-curves if necessary. 	The $x,y$-coordinates of $\beta$ are denoted by $x(t)=\langle \beta(t),e_1\rangle$ and $y(t)=\langle \beta(t),e_2\rangle$. 
	
	\vspace{0.3cm}
	
	We will first show monotonicity of the coordinates for $\beta$ in certain intervals. 
	Let  $t_1,\ldots,t_{n-1}$ be as in Definition \ref{def_nconvex}. If $\beta\big([0,t_1]\big)$ is a line segment, then $x(t), y(t)$ are monotonous, otherwise let $M_j,m_j\in \beta\big([0,t_1]\big)$ for $1\leq j\leq2$ satisfy
	$$\langle M_j,e_j\rangle=\max_{t\in [0,t_1]}\langle \beta(t),e_j\rangle
	\hspace{0.2cm} \mbox{ and }\hspace{0.2cm}
	\langle m_j,e_j\rangle=\min_{t\in [0,t_1]}\langle \beta(t),e_j\rangle.$$ 
	Choose  $o_j^M\in\beta^{-1}(M_j)\cap[0,t_1]$ and $o_j^m\in\beta^{-1}(m_j)\cap[0,t_1]$ for $j\in\{1,2\}$. Let $o_1,o_2,o_3,o_4$ denote these $o_j^m,o_j^M$ in ascending order. In the calculations below the pair $(\tau_1,\tau_2)$ is $(o_k,o_{k+1})$ for some $k\in\{0,1,2,3,4\}$ with $\tau_2-\tau_1>0$, where $o_0=0$ and $o_5=t_1$.

	It follows  by the convexity assumption that the $x,y$-coordinates of $\beta(t)$ are monotone for $t\in[\tau_1,\tau_2]$  as we will show: first,  for any $\kappa\in[0,t_1]$ and $\tau\in\{s\in[0,t_1]\ |\ \beta(s)\neq\beta(\kappa)\}$ we either have that the arc segment of $\beta$ with domain between $\tau$ and $\kappa$ is a straight line, or 
	$$\big\{ (1-s)\beta(\kappa)+s\beta(\tau)\ | \ 0<s<1 \big\}\cap \beta\big([0,t_1]\big)=\emptyset.$$
To prove monotonicity of $x(t)$ and $y(t)$,  we work out the example  $\beta(\tau_1)=m_2$.		 
	 The existence of numbers $\tau_1<\epsilon_1<\epsilon_2<\tau_2$ such that $y(\epsilon_1)>y(\epsilon_2)$ will lead to a contradiction in this case (by  the assumption on  $\beta(\tau_1)$, $y(t)$ is monotonously increasing): either  $x(\epsilon_1)\leq x(\epsilon_2)$ or $x(\epsilon_2)< x(\epsilon_1)$, in both cases, and depending if either $o^M_2<\tau_1$ or $o^M_2\geq \tau_2$ (by definition $o^M_2\notin(\tau_1,\tau_2)$), one of the  polygons with vertices and edges of the form  	 
	 \begin{align*}
	 	\beta(o^M_2)&\longrightarrow\beta(\tau_1)\longrightarrow\beta(\epsilon_1)\longrightarrow\beta(\epsilon_2),\\
	 	\beta(\tau_1)&\longrightarrow\beta(\epsilon_1)\longrightarrow\beta(\epsilon_2)\longrightarrow\beta(o^M_2)
	 \end{align*}
	 contradicts convexity, as we can find two points $A,B$ on the polygon, such that $\overrightarrow{AB}$ intersects it in a third point, and the same  will hence also be true for the arc segments $\beta\big([o^M_2,\tau_2]\big)$ or $\beta\big([\tau_1,o^M_2]\big)$;  (note that $y(o^M_2)\geq y(\epsilon_1)>y(\epsilon_2)\geq y(\tau_1)=y(o^m_2)$).

	 Thus  $y(t)$ is monotonously increasing, and we use this fact to show that $x$ is monotone. Assume there are numbers $\tau_1<\epsilon_1<\epsilon_2<\epsilon_3<\tau_2$ so that $x(\epsilon_2)< x(\epsilon_1)$ and $x(\epsilon_2)< x(\epsilon_3)$, and  if  either $o^*_1<\tau_1$ or $o^*_1\geq \tau_2$ (here $o^*_1$ is a place holder where $*\in\{m,M\}$ depending on the polygonal shape), one of the  polygons with vertices and edges of the form  
	 \begin{align*}
	 \beta(\tau_1)&\longrightarrow\beta(\epsilon_1)\longrightarrow\beta(\epsilon_2)\longrightarrow\beta(\epsilon_3)\longrightarrow\beta(o^*_1),\\
	 \beta(o^*_1)&\longrightarrow\beta(\tau_1)\longrightarrow\beta(\epsilon_1)\longrightarrow\beta(\epsilon_2)\longrightarrow\beta(\epsilon_3)
	 \end{align*}
	 contradicts convexity in a similar fashion as above.
	 The case: $x(\epsilon_1)< x(\epsilon_2)$ and $x(\epsilon_3)< x(\epsilon_2)$ is similar. Thus $x(t)$ is monotonous.	  
	  The other choices for $\tau_1\in\{o_1,\ldots,o_5\}$ reduce to the case above by applying rotations to $\beta\big([0,t_1]\big)$. 

\vspace{0.2cm}
	
	 We define two supporting axes  for each non-trivial arc $\beta([\tau_1,\tau_2])$ as follows: 	 
		$$
		\big\{ (1-t)\beta(\tau_j)+tZ \ |\ 0\leq t\leq 1 \big\}
	\hspace{0.3cm}\mbox{where}\hspace{0.3cm}
	Z=\binom{\langle \beta(\tau_1),e_1\rangle}{\langle \beta(\tau_2),e_2\rangle}\in\R^2.
	$$
	These line segments above are parallel to the axes, and we denote them accordingly by $L_x$ and $L_y$. Let $\gamma=\beta([\tau_1,\tau_2])$ and $Q$ be the identity matrix $\fatone$, then
	$$ I_\gamma^{\fatone}(K)\leq I_{L_x}^{\fatone}(K)+I_{L_y}^{\fatone}(K)$$
	by following argument: since $x(t),y(t)$ are monotonous  for $t\in[\tau_1,\tau_2]$, say both are increasing, then $\beta$ exits a fundamental domain $\frac{1}{K}\Omega^\fatone+\frac{1}{K}p$, where $p\in\Lambda^\fatone=\Z^2$, only by leaving it trough the top or right side.
	\begin{enumerate}
		\item 	If $\beta$ leaves through the right side, $I_\gamma^{\fatone}(K)$ and $I_{L_x}^{\fatone}(K)$ increase by one, 
		\item or else $I_\gamma^{\fatone}(K)$ and $I_{L_y}^{\fatone}(K)$ increase by one.
	\end{enumerate}
	 Thus, with $\mathrm{length}(L_x)=c$ and $\mathrm{length}(L_y)=d$, we have
	$$ I_\gamma^{\fatone}(K)\ \leq\  I_{L_x}^{\fatone}(K)+I_{L_y}^{\fatone}(K)\ \leq\  (c+d)K+4.$$
	We further use the inequality $c+d\leq \sqrt{2}\sqrt{c^2+d^2}$ valid for all $c,d\in\R$, to derive
	$$I_\gamma^{\fatone}(K)\ \leq\ \sqrt{2}\sqrt{c^2+d^2}\cdot K+4\ \leq\ \sqrt{2}\ \mathrm{length}(\gamma)\cdot K+4,$$
	where we used that the shortest path between $\beta(\tau_1)$ and $\beta(\tau_2)$ has length 	$\sqrt{c^2+d^2}$.
	The same reasoning applies to all sub-intervals (which are at most five), and after summing up and taking into account that every self-intersection counts a certain domain twice, we obtain 
	$$I_\beta^{\fatone}(K)\leq \sqrt{2}\  \mathrm{length}(\beta)\cdot K+19n+1-m.$$
	We will reduce the general case to the one above, by regarding the curve $\gamma=Q^{-1}\beta$. It is clear that $I_\gamma^{\fatone}(K)= I_\beta^{Q}(K)$.  Regularity of a curve is not affected by invertible matrices, neither is the notion of $n$-convexity nor the values $t_1,\ldots,t_n$.
\end{proof}
Note that the constant $\sqrt{2}$ in Lemma \ref{lem_sharplengthbound} cannot be improved, as the example of the translated diagonal of $I^2$ with $Q=\fatone$ already shows.

\subsection{Bound of $d^Q(K)$ and proof of the main result}
In this section we prove Theorem \ref{thm_main}, but first we bound the quantity $d^Q(K)$ and use this bound to express $N^Q(K)^{-\onehalf}$ in terms of $K$ and $Q$ in equation \eqref{eq_K2N}.

\begin{lemma}\label{lem_differenceKandN2}
	Given a  lattice $\Lambda^Q$ and $K\in\N$, then
	\begin{equation*}
	d^Q(K)\ \leq\ \sqrt{2}\cdot 2\cdot\big( \|Q^{-1}e_1\|+\|Q^{-1}e_2\|\big)  +\frac{20}{K}\ \leq\ \frac{4\ \|Q\|_F}{|\det(Q)|}  +\frac{20}{K}. 
	\end{equation*}
\end{lemma}
\begin{proof}	
	Using the notation as introduced in Definition \ref{def_pointsPNtilingTN}, we see that $|\det(Q)|K^{-2}$ is the area of each $T_K(p)$, with $p\in\Lambda^Q$. Now if we remove each $z^p_K\in P^Q(K)$ with $T_K(p)\subset I^2$, we can do this at most $\lfloor K^2 |\det(Q)|^{-1}\rfloor$ many times; which differs from $N^Q(K)$ by the amount of points $\{z^q_K\}$, where $T_K(q)$ intersects $\partial I^2$. The amount of these points is at most linear in $K$  by Lemma \ref{lem_sharplengthbound}, where the constant  can be computed explicitly since $\mathrm{image}(\beta)=\partial I^2$. The last part follows by applying the inequality $a+b\leq \sqrt{2}\sqrt{a^2+b^2}$ and the well known formula for $Q^{-1}$ in terms of elements of $Q$ and a factor of $\det(Q)^{-1}$.
\end{proof}
\begin{remark}
	The term $M^Q(K)$ can be bounded in the same fashion by the same constant $+O(K^{-1})$.
\end{remark}
For later use, we note that there is some $s_K$ with $|s_K|\leq\frac{d^Q(K)}{N^Q(K)}=O(K^{-2})$, such that
\begin{align}\label{eq_K2N}
	 \sqrt{\frac{1}{|\det(Q)|N^Q(K)}}=\frac{1}{K}\sqrt{1+\frac{K^2-|\det(Q)|N^Q(K)}{|\det(Q)|N^Q(K)}}= \frac{1}{K}+s_K.
\end{align}

\begin{proof}[Proof of Theorem \ref{thm_main}]
	Let $P^Q(K)=\{z_1,\ldots,z_N\}$ where $N=N^Q(K)$.
	Given a spherical cap $C(w,t)\subset\s$, we regard $A_{w,t}:=L^{-1}(C(w,t)\cap \s_p)\subset I^2$. Since $L$ is area preserving,
	 the proof will be complete once we can bound
	\begin{equation*}
	(\star):=	\Bigg| \frac{1}{N}\sum_{j=1}^{N}
		\chi_{A_{w,t}}(z_j)-\lambda(A_{w,t})\Bigg|. 
	\end{equation*}
	An upper bound of $(\star)$  is  established  by a classical approach\footnote{See Gauss circle problem, where Lemma \ref{lem_sharplengthbound}  could be applied. }, where $\Lambda^Q$ is split into sets $V,B$ of volume and boundary contributions  in order to apply the triangle inequality: let $\{p_1,\ldots,p_N\}\subset \Lambda^Q$ be such that $z_j\in T_K(p_j)$, then 
	$$(\star)\ \leq\ \sum_{ p\in V}\Big| \frac{1}{N}	-\lambda\big(T_K(p)\big)\Big| +\Big| \frac{1}{N}\sum_{ p_j\in B}	\chi_{A_{w,t}}(z_j)-\sum_{p_j\in B}\lambda\big(T_K(p_j)\cap A_{w,t}\big)\Big| , $$
	 where $V$ and $B$ are defined below, and where  the notation is as in Definition \ref{def_pointsPNtilingTN}. Set 
	 $$V=V(w,t,K)=\Big\{p\in\Lambda^Q\ |\  T_K(p)\subset A_{w,t} \Big\},$$
	hence $z^p_K\in A_{w,t}$ for $p\in V$, and we define the boundary term $B=B_C+B_I$ by
		$$B_C=B_C(w,t,K)=\Big\{q\in\Lambda^Q\ |\  T_K(q)\cap \big(\partial A_{w,t}\setminus\partial I^2\big)\neq\emptyset \Big\},$$
	
   and 
   \vspace{-.3cm}
			$$B_I=B_I(w,t,K)=\Big\{q\in\Lambda^Q\ |\  T_K(q)\cap \big(\partial A_{w,t}\cap\partial I^2\big)\neq\emptyset \Big\}.$$
	For each $p\in V$, the difference $\epsilon_K=N^{-1}-\lambda(T_K(p))$ is constant, thus the contribution to discrepancy coming from the area  is given by ($\#V\leq N$)
	$$|\epsilon_K|\cdot\#V\leq\Big|\frac{1}{N}-\frac{|\det(Q)|}{K^2} \Big|\cdot N=d^Q(K)\frac{|\det(Q)|}{KN}\cdot N.$$ 
    Next we bound the discrepancy coming from the boundary terms $B_C$ and $B_I$. 	
	If for $q\in B$  we have $z_K^q\notin A_{w,t}$, then  the difference inside the absolute value of $(\star)$
	has a negative contribution of $$\lambda\big(T_K(q)\cap A_{w,t}\big);$$
	 otherwise   the difference has a  contribution of $$\frac{1}{N}-\lambda\big(T_K(q)\cap A_{w,t}\big)=\lambda\big(T_K(q)\setminus A_{w,t}\big)+\epsilon_K.$$
	 Both contributions are hence bounded by $\lambda(T_K(p))$ up to $\epsilon_K$, and we obtain
	 $$ \sum_{ p\in B_C}\Big|\frac{\chi_{A_{w,t}}(z_K^p)}{N}	-\lambda\big(T_K(p)\cap A_{w,t}\big)\Big|\leq \frac{|\det(Q)|}{K}\sqrt{2}\cdot C^Q_L+O(K^{-2}),$$
	  where we applied Lemma  \ref{lem_sharplengthbound} (thanks to Corollary \ref{cor_AistleitnerEtAl}) and the definition of $C^Q_L$. Finally, the discrepancy coming from the contribution of points in $B_I$ is less than   $|\det(Q)|K^{-1}M^Q(K)$ by definition, and equation \eqref{eq_K2N} puts our bounds in terms of $N$.	
\end{proof}

\section{Applications to specific cases}\label{sec_applications}

In this section we first focus on point sets with special choice  $Q=\fatone$, where we obtain the lowest upper bound on the leading term of the spherical cap discrepancy   up to date, and the example also shows that in general the order of $O(N^{-\onehalf})$ cannot be improved. We further see that a separation distance of matching order cannot be derived in general.
   
   Next we describe an algorithm to modify a given point set $P^Q(K)$ to $P^Q_m(K)$, such that the values $d^Q_m(K), M^Q_m(K)$ related to $P^Q_m(K)$ are bounded by $\frac{10}{K}$. 
   \begin{wrapfigure}{r}{0.35\textwidth}
   	\begin{center}
   		\includegraphics[width=0.35\textwidth]{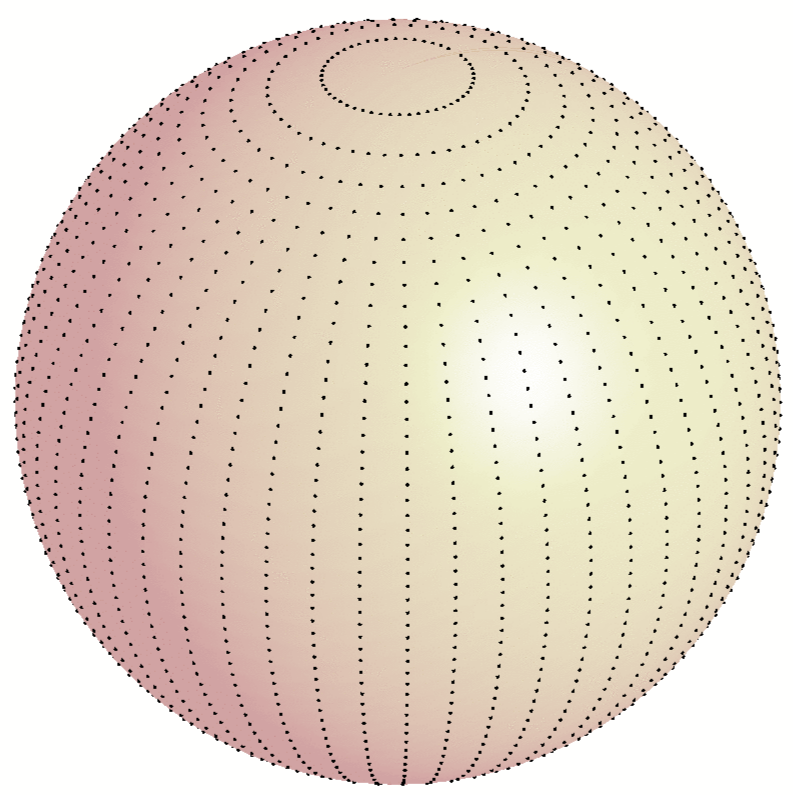}
   	\end{center}
   	\caption*{\textbf{Fig. 2:} Image of the points $P^\fatone(50)$  with $N=2500$.}
   	\vspace{-1cm}
   \end{wrapfigure}
   Lastly we give examples that are  visual evidence that some point sets of the type regarded in this article should do much better on all fronts.

	\subsection{The standard lattice}\label{subsec_standardlattice}
	Let $P^\fatone(K)=\big(\frac{1}{K}\Z^2+\frac{1}{2K}\binom{1}{1}\big)\cap I^2$. Then $d^\fatone(K)=0$ and $M^\fatone(K)=O(K^{-1})$, and we obtain  with Theorem $\ref{thm_main}$ and Lemma \ref{lem_boundonCL} the mentioned constant  in the introduction: 
	$$
	\frac{1}{2}\ \leq\ 	\sqrt{N}\cdot\sd(L(P^\fatone(K)))\ \leq\ \sqrt{18}\approx\ 4.242641.
	$$
	The lower bound comes from a special choice for a spherical cap: take as center the north pole and let the height $t\ra\frac{2K-1}{2K}^+$.

	The boundary of the aforementioned cap with height $t=\frac{2K-1}{2K}$ has length of order $K^{-1/2}$, while there are $K$ points equi-distributed on it, thus the distance between consecutive points  is of order $K^{-\frac{3}{2}}=N^{-\frac{3}{4}}$.

\subsection{Modifying point sets}\label{subsec_pointmod}

 
For $Q\in\gl$, choose $P^Q(K)$ as usual.  If $Ad([0,1],K)=\emptyset$ (recall the definition of $M^Q(K)$),
	then set $P^Q_{m}(K)=P^Q(K)$. 
	
	Otherwise we  define $P_m(K)$ by removing any $z_K^q\in P^Q(K)$ such that $q\in Ad([0,1],K)$. 
	Recall the parametrization $\rho$ of $\partial I^2$  from equation \eqref{eq_parametrizationI2}.
	
	 Let $L_t:=\big\{t_0,t_1,\ldots,t_R \big\}$ with $t_j<t_{j+1}$, be values such for each $j\in\{0,\ldots, R\}$, we have  $L_q:=\{q_0,\ldots,q_R\}\subset \Lambda^Q$ with $\rho(t_j)\in T_K(q_j)$. Let the choice be such, that $L_q$ is  maximal, i.e. no other $p\in\Lambda^Q\setminus L_q$ satisfies $\rho(\tau)\in T_K(p)$ for any $\tau$.
	 
	From these two lists $L_t,\ L_q$ we remove elements 	
	$t_{j_0}, q_{j_0}$ if 	
	 $q_{j_0}\notin Ad([0,1],K)$, to obtain reduced lists
	  $L'_t:=\{t_0',t_1',\ldots,t_{R'}'\}\subset L_t$, and  $L'_q:=\{q_0',\ldots,q_{R'}'\}\subset L_q$ with $R'\leq R$.
	   For $1\leq k\leq 4$ we collect the indices of the $t_j'$ belonging to the intervals $[k-1,k)$ into blocks
	  $$B_k=\big\{j\ |\ (k-1)\leq t_j'<k\big\}.$$
	  
	  Starting with $k=1$ we repeat following procedure: Let $s_1\in\N$ be the first instance such that  $\min(B_k)<s_1\leq \max(B_k)$ and	  
	  	$$  \frac{K^2}{|\det(Q)|}\sum_{j=\min(B_k)}^{s_1} \lambda\big(T_K(q_j')\cap I^2\big)>1.$$
	  	If there is no such instance, increase $k$ by one (up to $k=4$) and repeat the process.  Otherwise choose an arbitrary $\zeta_K^1\in T_K(q_{s_1}')\cap I^2$ and add it to $P_m^Q(K)$. Assume $\zeta^1_K, \ldots,\zeta_K^m$ have been chosen (and added to $P_m^Q(K)$), then let $s_{m+1}\in\N$ be the first instance such that $s_m<s_{m+1}\leq \max(B_k)$ and 	
	  	$$  \frac{K^2}{|\det(Q)|}\sum_{j=\min(B_k)}^{s_{m+1}} \lambda\big(T_K(q_j')\cap I^2\big)>m+1.$$
	  	If there is no such instance,  increase $k$ by one (up to $k=4$) and repeat the process. Otherwise choose an arbitrary $\zeta_K^{m+1}\in T_K(q_{s_{m+1}}')\cap I^2$ and add to $P_m^Q(K)$. Continue as long as possible -- this finishes the construction of $P_m^Q(K)$, and  the bound of $d^Q_m(K),M^Q_m(K)$ by $10\cdot K^{-1}$ is then evident.

\subsection{Orthonormal lattices}

 We choose points $P^{Q(x,y)}(K)$, where $Q$ is chosen orthonormal up to a factor, i.e. 
	$$
Q(x,y)=\frac{1}{y}\begin{pmatrix}
x & -1\\
1 & x
\end{pmatrix},\  
Q^{-1}(x,y)=\frac{y}{x^2+1}\begin{pmatrix}
x & 1\\
-1 & x
\end{pmatrix},
$$
\begin{wrapfigure}{r}{0.37\textwidth}
	\vspace{-0.7cm}
	\begin{center}
		\includegraphics[width=0.35\textwidth]{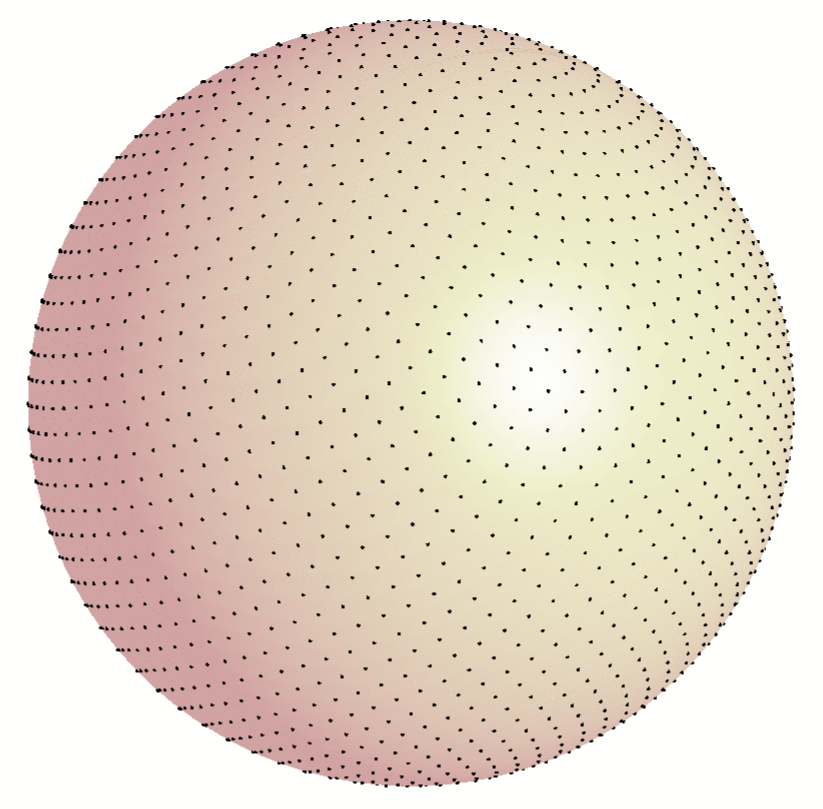}
	\end{center}
	\vspace{-0.3cm}
	\caption*{\textbf{Fig. 3:} Image of points related to the lattice $\Lambda^{Q(\varphi,1)}$ with $K=50$ and  $N^{Q}=2500$.}
	\vspace{-.5cm}	
\end{wrapfigure}
for $x\in\R$, $y>0$. Then, by Lemma \ref{lem_lengthboundcurvematrix}
$$C^{Q(x,y)}_L\leq 3\cdot \|Q^{-1}(x,y)e_1\|=3\frac{y}{\sqrt{x^2+1}}$$
 and 
$y\cdot \sqrt{\det(Q(x,y))}=\sqrt{x^2+1}.$
We  modify $P^{Q(x,y)}(K)$ as in Section \ref{subsec_pointmod}, and let $N_m^{Q(x,y)}=\#P^{Q(x,y)}(K)$, 
 so that  by Theorem \ref{thm_main}, we obtain  many deterministic point sets with the bound
$$
	\sd(L(P_{mod}^{Q(x,y)}(K)))\ \leq\  \sqrt{\frac{18}{N^{Q(x,y)}_{m}}}+O\Big(\frac{1}{K^2}\Big).
$$
For  $\varphi=\frac{1+\sqrt{5}}{2}$, the choice $Q(\varphi,1)$ yields the image in Figure 3, resembling spherical Fibonacci lattices and grids as in \cite{Aistleitner}, \cite{SwinbankPurser}. 

Note that no modification was necessary for $P^{Q(\varphi,1)}(K)$ -- this seems to be related to directional discrepancy as in \cite{BilkMaSpencer}, but we do not pursue this direction.

\begin{remark}
	Definition \ref{def_nconvex} and Lemma \ref{lem_sharplengthbound} were first developed by this author to prove the result in the paper \cite{FerHofMas} with Julian Hofstadler and Michelle Mastrianni, but the items were improved and streamlined in the current work, so the proofs will be found in both papers to make reviewing easier.
\end{remark}

\begin{remark}
	Clearly the proof idea of Theorem \ref{thm_main} extends beyond the Lambert projection to any area preserving map $\Gamma$  between a surface and shape $R\subset\R^2$ with $\partial R$ being $n$-convex -- as long as the boundary of the analog of spherical cap under $\Gamma$ is $n$-convex with universally bounded $n$ and length. Thus in \cite{FerHofMas}, one could use the HEALPix projection to derive a similar result.
\end{remark}
 
 \begin{acknowledgement}
The author thanks Dmitriy Bilyk 
and Arno Kuijlaars for skimming the text and useful remarks. The author further thanks the people involved with GNU Octave, LibreOffice and TeXstudio, which made this document possible.
\end{acknowledgement}

\end{document}